\documentclass[11pt,a4paper]{article}
\usepackage{}
\usepackage{amsfonts}
\usepackage{amssymb}
\usepackage{mathrsfs}
\usepackage{amsmath}
\usepackage{booktabs}
\usepackage{epsf,epsfig,amsfonts,amsgen,indentfirst}
\usepackage{amsmath,amstext,amsbsy,amsopn,amsthm,bbding,wasysym}
\usepackage{multicol,mathdots}

\newtheorem{theorem}{Theorem}[section]

\newtheorem{lemma}{Lemma}[section]
\newtheorem{cor}[theorem]{Corollary}

\newtheorem{definition}{Definition}[section]

\newtheorem{remark}{Remark}

\begin{document}
\title
{\LARGE \textbf{ Extremal  matching energy of complements of trees }}

\author{{\small Tingzeng Wu$^{a,}$\thanks{
Corresponding author.\newline
 \emph{ E-mail address}: mathtzwu@163.com (T. Wu), weigenyan@263.net (W. Yan), zhanghp@lzu.edu.cn (H. Zhang).\newline
 $^{1}$Partially supported by NSFC (11371180).\newline
 $^{2}$Partially supported by NSFC (11171134).} $^{, 1}$, Weigen Yan$^{b,2}$, Heping Zhang$^{c,1}$} \\
{\small $^{a}$School of Mathematics and Statistics, Qinghai Nationalities University,}\\
{\small  Xining, Qinghai 810007, P. R.~China}\\
{\small $^{b}$School of Sciences, Jimei University, Xiamen 361021, P. R.~China}\\
{\small $^{c}$School of Mathematics and Statistics, Lanzhou University,}\\
{\small Lanzhou, Gansu 730000, P. R.~China}}

\date{}

\maketitle
\noindent {\bf Abstract:}\ \ The matching energy is defined as the sum of the absolute values of
the zeros of the matching polynomial of a graph, which is proposed first by Gutman and Wagner [The matching energy of a graph, Discrete Appl. Math. 160 (2012) 2177--2187]. And they gave some properties and  asymptotic results of the
matching energy. In this paper,  we characterize the trees with $n$ vertices whose complements have the maximal, second-maximal and minimal matching energy. Further,  we   determine the trees with a perfect matching whose complements have the  second-maximal matching energy. In particular, show that the trees with  edge-independence number number $p$ whose complements have the minimum matching energy for $p=1,2,\ldots, \lfloor\frac{n}{2}\rfloor$.


\section{Introduction}
All graphs only considered in this paper are undirected simple graphs. For notation and terminologies not defined here, see \cite{god}. Let $G=(V(G),E(G))$ be a graph with the vertex set $V(G)=\{v_{1},v_{2},...,v_{n}\}$ and the edge set $E(G)=\{e_{1},e_{2},...,e_{m}\}$. Denote by $G-v$ or $G-e$ the graphs obtained from $G$ by removing $v$
or $e$, respectively, where $v\in V(G)$ and $e\in E(G)$. Denote by $\overline{G}$ the complement of $G$. The path, star and complete graph with $n$ vertices are denoted by $P_{n}$, $K_{1,n-1}$ and $K_{n}$, respectively. Let $T_{n,2}$ be a tree obtained from the star $K_{1,3}$ by attaching a path $P_{n-3}$ to one of the pendent vertices of $K_{1,3}$, and  let $T_{n,2}^{1}$
 be a tree obtained from the star $K_{1,3}$ by attaching two paths $P_{2}$ and $P_{n-4}$ to two of the different pendent vertices of $K_{1,3}$ respectively.
Let $T_{n}^{p}$ be a tree with $n$ vertices obtained from the star $K_{1,n-p}$ by attaching a pendent edge
to each of $p-1$ pendent vertices in $K_{1,n-p}$ for $p = 1,2,...,\lfloor \frac{n}{2}\rfloor$.

A $k$-matching in $G$ is a set
of $k$ pairwise non-incident edges. The number of $k$-matchings in $G$ is denoted by $m(G, k)$. Specifically, $m(G,0)=1$, $m(G,1)=m$ and $m(G,k)=0$ for $k>\frac{n}{2}$ or $k<0$. For a $k$-matching $M$ in $G$, if $G$ has no $k'$-matching  such that $k'>k$,  then $M$ is called  a {\em maximum matching} of $G$. The number $\nu(G)$ of edges in a maximum matching $M$ is called {\em the edge-independence number} of $G$.
Let $\mathcal{T}_{n,p}$ denote the set of trees with $n$ vertices and the edge-independence number at least $p$ for $p=1,2,...,\lfloor\frac{n}{2}\rfloor$.
The {\em Hosoya index} $Z(G)$ is defined as the total number of matchings of $G$, that is
 $$Z(G)=\sum_{k=0}^{\lfloor\frac{n}{2}\rfloor}m(G,k).$$
\indent  Recall that for a graph $G$ on $n$ vertices, the {\em matching polynomial} $\mu(G,x)$ of $G$ (sometimes denoted by $\mu(G)$ with no confusion) is given by
 $$\mu(G,x)=\sum_{k\geq 0}(-1)^{k}m(G,k)x^{n-2k}.\eqno(1)$$
Its theory is well elaborated \cite{cve1,far,god,god1,gut}. Gutman and Wagner \cite{gut1}
gave the definition  of  the {\em quasi-order} $\succeq$ as follows. If $G$ and $H$ have the matching polynomials in the form (1), then the quasi-order $\succeq$ is defined by
$$G\succeq H\Longleftrightarrow m(G,k)\geq m(H,k)\quad \text{for all} \ \ k=0,1,...,\lfloor n/2\rfloor. \eqno(2)$$
Particularly, if $G\succeq H$ and there exists some $k$ such that $m(G,k)> m(H,k)$,
then we write $G\succ H$.\\
\indent Gutman  and Wagner in \cite{gut1} first proposed the  concept of the {\em matching energy} of a graph, denoted by $ME(G)$, as $$ME=ME(G)=\frac{2}{\pi}\int_{0}^{\infty}x^{-2}{\rm ln} \Bigg[\sum_{k\geq0}m(G,k)x^{2k}\Bigg]dx.\eqno(3)$$
Meanwhile, they gave also an other  form of definition of   matching energy of a graph. That is, $$ME(G)=\sum_{i=1}^{n}|\mu_{i}|,$$
where $\mu_{i}$ denotes the root  of  matching polynomial of $G$.
Additionally, they found some relations between the matching energy and  energy (or reference energy). By (2) and (3), we easily obtain the  fact as follows.
$$G\succeq H\Longrightarrow ME(G)\geq ME(H)\quad \text{and}\quad G\succ H\Longrightarrow ME(G)> ME(H). \eqno(4)$$
This property is an important technique to determine extremal graphs with the matching energy.\\
\indent Note that the energy (or reference energy) of graphs are extensively examined (see \cite{aih,cve1,dew,gut2,gut3,lxl}). However, the literatures on
the matching energy are far less than that on the energy and reference energy. Up to now, we find only few papers about the matching energy  published. Gutman and Wagner \cite{gut1} gave some properties and asymptotic results of the matching energy. Li and Yan \cite{li} characterized the connected graph with the fixed connectivity (resp. chromatic number) which has the maximum matching energy. Ji et al. in \cite{ji} determined completely the graphs with the minimal and maximal matching energies in bicyclic graphs. Li et al. \cite{lih} characterized  the unicyclic graphs with fixed girth (resp.  clique number) which has the maximum
and minimum matching energy.\\
\indent In this paper, inspired by the idea \cite{yan}, we investigate the problem of the matching energy of the complement of trees and obtain the following  main theorems.
\begin{theorem}\label{art01}
Let  $T$ be a tree with $n$ vertices. If $T\ncong T_{n,2}$ and $T\ncong P_{n}$, then $$ME(\overline{T})<ME(\overline{T_{n,2}})<ME(\overline{P_{n}}).$$
\end{theorem}
\begin{theorem}\label{art02}
Let  $\mathcal{T}_{n,p}$ denote the set of trees with $n$ vertices and the edge-independence number at least $p$ for $p=1,2,...,\lfloor\frac{n}{2}\rfloor$. For a tree $T\in\mathcal{T}_{n,p}$, then $$ ME(\overline{T})\geq ME(\overline{T_{n}^{p}})$$
with equality if and only if  $T\cong T_{n}^{p}$.
\end{theorem}
 By Theorems \ref{art01} and \ref{art02}, we obtain directly the following corollary.
\begin{cor}\label{art03}
 The components of $P_{n}$ and $K_{1,n-1}$ have the maximum and minimum matching energy in all components of trees, respectively.
\end{cor}
\begin{theorem}\label{art04}
Let $\mathscr{T}_{n,\frac{n}{2}}$  be a proper subset of $\mathcal{T}_{n,p}$ containing  all trees with a perfect matching.  Suppose that $T\in \mathscr{T}_{n,\frac{n}{2}}$, $T\neq T_{n}^{\frac{n}{2}}$ and $P_{n}$.  If $n\geq 6$, then
$$ME(\overline{T_{n}^\frac{n}{2}})<ME(\overline{T})\leq ME(\overline{T_{n,2}^{1}})<ME(\overline{P_{n}}),$$
where the equality holds if and only if $T\cong T_{n,2}^{1}$.
\end{theorem}

\section{Some Lemmas}
There exists a well-known formula which characterizes the relation between $m(G,r)$ and $m(\overline{G},i)$ (see Lov\'{a}sz \cite{lov}), which will play a key role in the proofs of the main theorems.

\begin{lemma}\label{art21}{\rm(\cite{lov}) }
Let $G$ be a simple graph with $n$ vertices and $\overline{G}$ the complement of $G$. Then
$$m(G,r)=\sum_{i=0}^{r}(-1)^{i}{n-2i\choose 2r-2i}(2r-2i-1)!!m(\overline{G},i),\eqno(5)$$
where $s!!=s\times(s-2)!!$, and $(-1)!!=0!!=1$.
\end{lemma}

The following results about the matching polynomial of $G$ can be found in Godsil \cite{god}.
\begin{lemma}\label{art22}{\rm(\cite{god}) }
The matching polynomial satisfies the following identities:\\
$(i)$ $\mu(G\cup H, x)=\mu(G,x)\mu(H,x),$ \\
$(ii)$ $\mu(G, x)=\mu(G\setminus e,x)-\mu(G-u-v,x)$ if $e=\{u, v\}$ is an edge of $G$,\\
$(iii)$ $\mu(G, x)=x\mu(G\setminus u,x)-\sum_{v\thicksim u}\mu(G-u-v,x)$ if $u\in V(G)$.
\end{lemma}

\begin{lemma}\label{art23}{\rm(\cite{god}) }
Let $m$ and $n$ be two positive integers. Then
$$\mu(P_{m+n})=\mu(P_{m})\mu(P_{n})-\mu(P_{m-1})\mu(P_{n-1}).\eqno(6)$$
\end{lemma}
\begin{lemma}\label{art24}{\rm(\cite{yan}) }
If $T$ is a tree with $n$ vertices and
edge-independence number $\nu(T)=p$, then $T$ has at
most $n-p$ vertices of degree one. In particular, if T
has exactly $n-p$ vertices of degree one, then every
vertex of degree at least two in $T$ is adjacent to at
least one vertex of degree one.
\end{lemma}

\section{Ordering complements of trees with respect to their matchings}
 For convenience,  we use the same definitions of trees which are defined in \cite{yan}.
\begin{definition}\label{art31}
Let $T_{1}$  be a tree with $n+m+k$ vertices shown in Figure \ref{fig1}, where $T_{0}$ is a tree with $k$ vertices ($k\geq2$) and $u$ a vertex of $T_{0}$, $n\geq1$ and  $m \geq1$. Suppose $T_{2}$ is a tree with $n+m+k$
vertices obtained from $T_{0}$ by attaching a path $P_{m+n}$ to u in $T_{0}$ (see Figure \ref{fig1}).
We designate the transformation from $T_{1}$ to $T_{2}$ as of type {\bf 1} and denote it by $\mathcal{F}_{1}$: $T_{1}\hookrightarrow T_{2}$ or $\mathcal{F}_{1}(T_{1})=T_{2}$.
\end{definition}
 \begin{figure}[htbp]
\begin{center}
\includegraphics[scale=0.6]{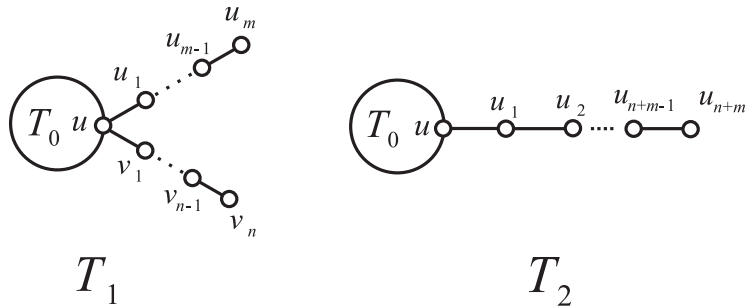}
\caption{\label{fig1}\small
{Two trees $T_{1}$ and $T_{2}$.}}
\end{center}
\end{figure}

\begin{theorem}\label{art32}
Let  $T_{1}$ and $T_{2}$ be the
trees with $m+n+k$ vertices defined in Definition \ref{art31}. Then $\overline{T_{2}}\succ\overline{T_{1}} $.
\end{theorem}
\begin{proof}
By  Lemma \ref{art22},\\
\begin{equation*}
\begin{split}
\mu(T_{1})&=x\mu(T_{0}-u)\mu(P_{m})\mu(P_{n})-\mu(T_{0}-u)\mu(P_{m-1})\mu(P_{n})
-\mu(T_{0}-u)\mu(P_{m})\mu(P_{n-1})\\
&\quad -\sum_{v\in V(T_{0})\atop uv\in E(T_{0})}\mu(T_{0}-u-v)\mu(P_{m})\mu(P_{n}),\\
\mu(T_{2})&=x\mu(T_{0}-u)\mu(P_{m+n})-\mu(T_{0}-u)\mu(P_{m+n-1})
-\sum_{v\in V(T_{0})\atop uv\in E(T_{0})}\mu(T_{0}-u-v)\mu(P_{m+n}),
\end{split}
\end{equation*}
where the above sums range over all vertices of $T_{0}$ adjacent to $u$. Hence
\begin{equation*}
\begin{split}
\mu(T_{1})-\mu(T_{2})&=x\mu(T_{0}-u)[\mu(P_{m})\mu(P_{n})-\mu(P_{m+n})]-\mu(T_{0}-u)
[\mu(P_{m-1})\mu(P_{n})-\mu(P_{m+n-1})\\
&\quad +\mu(P_{m})\mu(P_{n-1})]-\left[\mu(P_{m})\mu(P_{n})-\mu(P_{m+n})\right]\sum_{v\in V(T_{0})\atop uv\in E(T_{0})}\mu(T_{0}-u-v)
\end{split}
\end{equation*}
By (6) and a routine calculation,
$$
\mu(T_{1})-\mu(T_{2})=-\sum_{v\in V(T_{0})\atop uv\in E(T_{0})}\mu(T_{0}-u-v)\mu(P_{m-1})\mu(P_{n-1}).\eqno(7)
$$
For an arbitrary vertex $v$ adjacent to $u$ in $T_{0}$, let $T_{v}^{*}$ be the forest $(T_{0}-u-v)\cup P_{m-1}\cup P_{n-1}$, which has $n+m+k-4$ vertices. By (5), we obtain
$$m(\overline{T_{1}},r)-m(\overline{T_{2}},r)=\sum_{i=0}^{r}(-1)^{i}{n+m+k-2i\choose 2r-2i}(2r-2i-1)!![m(T_{1},i)-m(T_{2},i)].\eqno(8)$$
Note that $m(T_1,0)=m(T_2,0)$ and $m(T_1,1)=m(T_2,1)$. Hence
$$
m(\overline{T_{1}},r)-m(\overline{T_{2}},r)=-\sum_{v\in V(T_{0})\atop uv\in E(T_{0})}\sum_{i=2}^r(-1)^{i}{n+m+k-2i\choose 2r-2i}(2r-2i-1)!!m(T_{v}^{*},i-2). \eqno{(9)}
$$
Note that $T_{v}^{*}$ has $n+m+k-4$ vertices. So
$$m(\overline {T_v^*},r-2)=\sum_{j=0}^{r-2}(-1)^{j}{n+m+k-4-2j\choose 2(r-2)-2j}(2(r-2)-2j-1)!!m(T_{v}^{*},j)$$
$$=\sum_{i=2}^r(-1)^{i}{n+m+k-2i\choose 2r-2i}(2r-2i-1)!!m(T_{v}^{*},i-2).$$
Hence
$$
m(\overline{T_{1}},r)-m(\overline{T_{2}},r)=-\sum_{v\in V(T_{0})\atop uv\in E(T_{0})}m(\overline{T_{v}^{*}},r-2).\eqno(10)
$$
By the definition of $m(G,r)$ and (10), we have $m(\overline{T_{v}^{*}},r-2)\geq 0$, which implies
 $m(\overline{T_{1}},r)\leq m(\overline{T_{2}},r)$. Particularly, if $r=2$, then $m(\overline{T_{1}},r)-m(\overline{T_{2}},r)\leq -1$. By (2),  $\overline{T_{2}}\succ\overline{T_{1}}$.
\end{proof}

\begin{remark}
By Theorem \ref{art32} and (4), we obtain immediately a result as follows: If $T_{1}$ and $T_{2}$ are the two
trees defined in Definition \ref{art31}, then $ME(\overline{T_{2}})>ME(\overline{T_{1}}) $. Additionally,  by the definition of the Hosoya index and Theorem \ref{art32}, it is not difficult  to see that $Z(\overline{T_{2}})>Z(\overline{T_{1}})$.
\end{remark}

\begin{definition}\label{art33}
 Let $T_{3}$ and $T_{4}$ be two trees with $m+n+s+1$ vertices shown in Fig. \ref{fig2}, where $s\geq m\geq 2, n \geq 1$.
We designate the transformation from $T_{3}$ to $T_{4}$ in Figure \ref{fig2} as of type {\bf 2} and denote it by $\mathcal{F}_{2}$: $T_{3}\mapsto T_{4}$ or $\mathcal{F}_{2}(T_{3})=T_{4}$.
\end{definition}
 \begin{figure}[htbp]
\begin{center}
\includegraphics[scale=0.60]{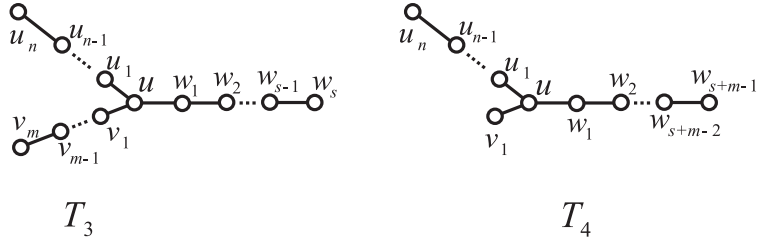}
\caption{\label{fig2}\small
{Two trees $T_{3}$ and $T_{4}$.}}
\end{center}
\end{figure}

\begin{theorem}\label{art34}
Let $T_{3}$ and $T_{4}$ be two trees with $m+n+s+1$ vertices defined in Definition \ref{art33}. Then $\overline{T_{4}}\succ\overline{T_{3}}$.
\end{theorem}
\begin{proof}
Similar to the proof of Theorem \ref{art32}, we can obtain that
$$\mu(T_{3})-\mu(T_{4})=-\mu(P_{m-2})\mu(P_{n-1})\mu(P_{s-2}).$$
Furthermore, we also have
$$
m(\overline{T_{3}},r)-m(\overline{T_{4}},r)=-m(\overline{P_{m-2}\cup P_{n-1}\cup P_{s-2}},r-3).\eqno(11)$$
\indent By the definition of $m(G,r)$ and (11), we have $m(\overline{P_{m-2}\cup P_{n-1}\cup P_{s-2}},r-3)\geq0$, which implies $m(\overline{T_{3}},r)\leq m(\overline{T_{4}},r)$. Specially, if $r=3$ then $m(\overline{P_{m-2}\cup P_{n-1}\cup P_{s-2}},r-3)=1$. This means, by (2), that
  $\overline{T_{4}}\succ\overline{T_{3}}$.
The proof is complete.
\end{proof}

\begin{definition}\label{art35} Let $T_{5}$ and $T_{6}$ be two trees with $m+n+2$ vertices shown in Fig. \ref{fig3}, where $m\geq n\geq 2$.
We designate the transformation from $T_{5}$ to $T_{6}$ in Figure \ref{fig3} as of type {\bf 3} and denote it by $\mathcal{F}_{3}$: $T_{5}\rightarrow T_{6}$ or $\mathcal{F}_{3}(T_{5})=T_{6}$.
\end{definition}
\begin{figure}[htbp]
\begin{center}
\includegraphics[scale=0.50]{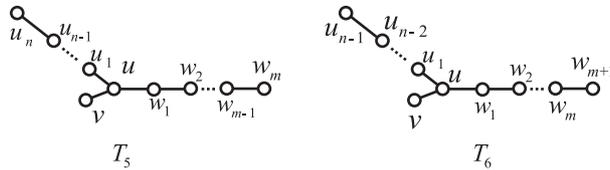}
\caption{\label{fig3}\small
{Two trees $T_{5}$ and $T_{6}$.}}
\end{center}
\end{figure}
\begin{theorem}\label{art36}
Let $T_{5}$ and $T_{6}$ be two trees with $m+n+2$ vertices defined in Definition \ref{art35}. Then $\overline{T_{6}}\succ\overline{T_{5}}$.
\end{theorem}

\begin{proof}
Similar to the proof of Theorem \ref{art31}, we have
$$\mu(T_{5})-\mu(T_{6})=-\mu(P_{m-n})$$
and
$$m(\overline{T_{5}},r)-m(\overline{T_{6}},r)=-m(\overline{P_{m-n}},r-n-1).\eqno(12)$$

\indent
By the definition of $m(G,r)$ and (12), we have $m(\overline{P_{m-n}},r-n-1)\geq0$, which indicates $m(\overline{T_{5}},r)\leq m(\overline{T_{6}},r)$. Specially,  when $r= n+1$, then $m(\overline{P_{m-n}},r-n-1)= 1$. By (2), we get that $\overline{T_{4}}\succ\overline{T_{3}}$.
\end{proof}

\begin{definition}\label{art37}
Suppose that $T'_{1}$ and $T'_{2}$ are two trees with $m$ ($m > 1$) vertices
and with $n$ ($n > 1$) vertices, respectively. Take one vertex $u$ of $T'_{1}$ and one $v$ of $T'_{2}$.
Construct two trees $T_{7}$ and $T_{8}$ with $m+n$ vertices as follows. The vertex set $V (T_{7})$ of
$T_{7}$ is $V (T'_{1})\cup V (T'_{2})$ and the edge set of $T_{7}$ is $E(T'_{1} )\cup E(T'_{2})\cup {uv}$. $T_{8}$
is the tree obtained from $T'_{1}$ and $T'_{2}$ by identifying the vertex $u$ of $T'_{1}$ and the vertex $v$ of
 $T'_{2}$ and adding a pendent edge $uw=vw$ to this new vertex $u$ ($=v$). The result graphs see Fig. \ref{fig4}.
 We
designate the transformation from $T_{7}$ to $T_{8}$ as of type {\bf 4} and denote it by $\mathcal{F}_{4}$: $T_{7}\looparrowright  T_{8}$ or $\mathcal{F}_{4}(T_{7})=T_{8}$.
\end{definition}
 \begin{figure}[htbp]
\begin{center}
\includegraphics[scale=1.20]{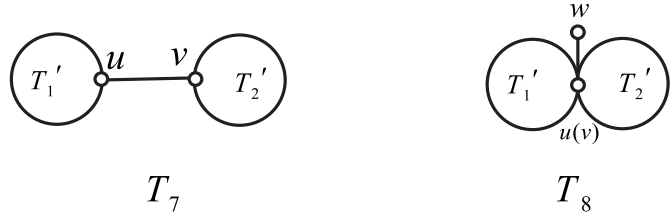}
\caption{\label{fig4}\small
{Two trees $T_{7}$ and $T_{8}$.}}
\end{center}
\end{figure}
\begin{figure}[htbp]
\begin{center}
\includegraphics[scale=0.60]{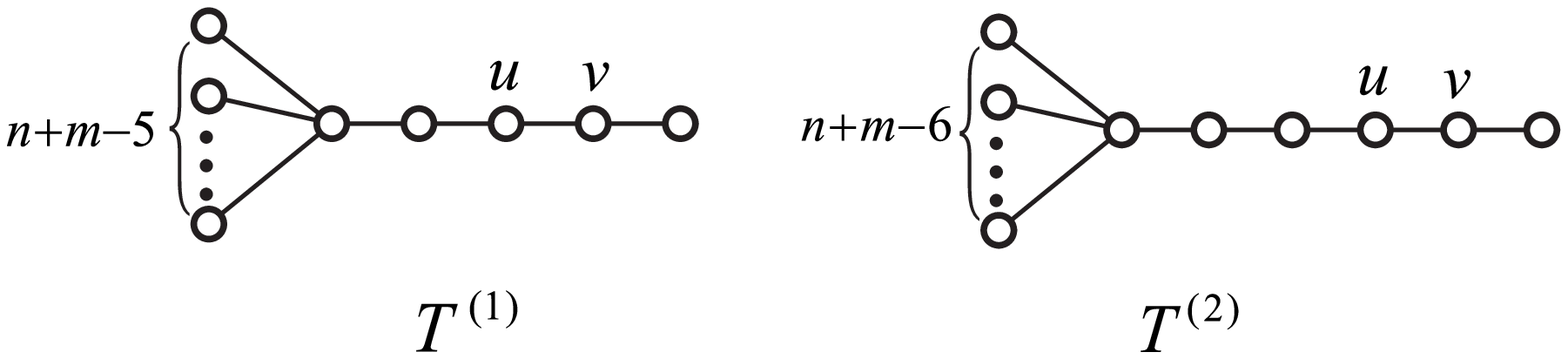}
\caption{\label{fig41}\small
{Two trees $T^{(1)}$ and $T^{(2)}$.}}
\end{center}
\end{figure}

\begin{theorem}\label{art38}
Let $T_{7}$ and $T_{8}$ be two trees with $m+n$ vertices defined in
Definition \ref{art37}. Then $\overline{T_{7}}\succ \overline{T_{8}}$.
\end{theorem}
\begin{proof}
By Lemma \ref{art22},
$$\mu(T_{7})=\mu(T'_{1})\mu(T'_{2})-\mu(T'_{1}-u)\mu(T'_{2}-v),\eqno(13)$$
$$\mu(T_{8})=x\mu(T_{8}-w)-\mu(T'_{1}-u)\mu(T'_{2}-v),\eqno(14)$$
$$\mu(T'_{1})=x\mu(T'_{1}-u)-\sum_{i=1}^{s}\mu(T'_{1}-u-u_{i})\eqno(15)$$
and
$$\mu(T'_{2})=x\mu(T'_{2}-v)-\sum_{j=1}^{s}\mu(T'_{2}-v-v_{j}).\eqno(16)$$
where the first sum ranges over all vertices $u_{i}$ ($1\leq i \leq s$) of $T'_{1}$ adjacent to $u$
and the second sum ranges over all $v_{j}$ ($1\leq j \leq t$) of $T'_{2}$ adjacent to $v$.
By (15) and (16), we have
\begin{equation*}
\begin{split}
x\mu(T_{8}-w)&=x^{2}\mu(T'_{1}-u)\mu(T'_{2}-v)-x\sum_{j=1}^{t}\mu(T'_{1}-u)\mu(T'_{2}-v-v_{j})\\
&\quad -x\sum_{i=1}^{s}\mu(T'_{2}-v)\mu(T'_{1}-u-u_{i}) \qquad  \qquad \qquad \qquad \qquad \qquad \qquad \qquad \qquad \qquad (17)
\end{split}
\end{equation*}
and
\begin{equation*}
\begin{split}
\mu(T'_{1})\mu(T'_{2})&=x^{2}\mu(T'_{1}-u)\mu(T'_{2}-v)-x\sum_{j=1}^{t}\mu(T'_{1}-u)\mu(T'_{2}-v-v_{j})\\
&\quad -x\sum_{i=1}^{s}\mu(T'_{2}-v)\mu(T'_{1}-u-u_{i})+\sum_{1\leq i\leq s \atop 1\leq j\leq t}\mu(T'_{1}-u-u_{i})\mu(T'_{2}-v-v_{j}).  \qquad \quad \ \ (18)
\end{split}
\end{equation*}
Combining (13), (14), (17) and (18),
$$\mu(T_{7})-\mu(T_{8})=\sum_{1\leq i\leq s \atop 1\leq j\leq t}\mu(T'_{1}-u-u_{i})\mu(T'_{2}-v-v_{j}).\eqno(19)$$
\indent As in the proof of Theorem \ref{art32}, we can show that
$$m(\overline{T_{7}},r)-m(\overline{T_{8}},r)=\sum_{1\leq i\leq s \atop 1\leq j\leq t}m(\overline{\mu(T'_{1}-u-u_{i})\cup \mu(T'_{2}-v-v_{j})},r-2),$$
which implies that
$$m(\overline{T_{7}},r)\geq m(\overline{T_{8}},r).$$
Note that $m(\overline{T_{7}},r)- m(\overline{T_{8}},r)\geq1$ when $r=2$. So, by (2), the theorem holds.
\end{proof}

\begin{remark}
For the trees Fig.\ref{fig41}, we note that neither tree $T^{(1)}$ nor tree $T^{(2)}$ can be transformed into $T^{p}_{m+n}$ by a single transformation {\bf 4}. Hence if $T_{8}$ in Theorem \ref{art38} is $T^{p}_{m+n}$, then $\overline{T_{7}}\succ\overline{T_{8}}=\overline{T^{p}_{m+n}}$. Particularly,
$\overline{T^{p}_{n}}\succ\overline{T^{p-1}_{n}}$ for $n\geq5$. Similarly, it is easy to show that the statement holds.
\end{remark}

\begin{definition}\label{art39}
 Suppose that $T_{9}$ is a tree with $n$ vertices and with the edge-independence
number $p$ shown in Fig. \ref{fig5} which has exactly $n-p$ pendent vertices, where
$|V(T_{0})|\geq2$ and $r\geq2$. Let $T_{10}$ be the tree with $n$ vertices shown in Fig. \ref{fig5},
which is obtained from $T_{9}$.
 We designate the transformation from $T_{9}$ to $T_{10}$ as of type
{\bf 5} and denote it by $\mathcal{F}_{5}$: $T_{9}\dashrightarrow  T_{10}$ or $\mathcal{F}_{5}(T_{9})=T_{10}$.
\end{definition}
 \begin{figure}[htbp]
\begin{center}
\includegraphics[scale=1.0]{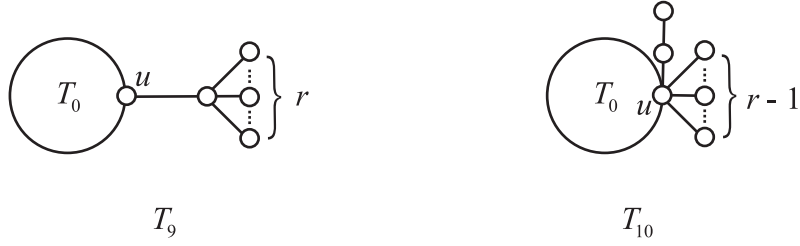}
\caption{\label{fig5}\small
{Two trees $T_{9}$ and $T_{10}$.}}
\end{center}
\end{figure}

\begin{theorem}\label{art310}
Let $T_{9}$ and $T_{10}$ be two trees with $n$ vertices defined in
Definition \ref{art39}. Then  $\overline{T_{9}}\succ \overline{T_{10}}$.
\end{theorem}
\begin{proof}
By Lemma \ref{art22},
\begin{equation*}
\begin{split}
\mu(T_{9})&=x\mu(T_{0}-u)\mu(K_{1,r})-\mu(K_{1,r})\sum_{v\in V(T_{0})\atop uv\in  E(T_{0}) }\mu(T_{0}-u-v)-\mu(T_{0}-u)\mu(P_{1})^{r}\\
&=x^{2}\mu(T_{0}-u)\mu(P_{1})^{r}-rx\mu(T_{0}-u)\mu(P_{1})^{r-1}-x\mu(P_{1})^{r}\sum_{v\in V(T_{0})\atop uv\in  E(T_{0}) }\mu(T_{0}-u-v)\\
&\quad+r\mu(P_{1})^{r-1}\sum_{v\in V(T_{0})\atop uv\in  E(T_{0}) }\mu(T_{0}-u-v)
 -\mu(T_{0}-u)\mu(P_{1})^{r}\qquad \qquad \qquad \qquad \qquad \qquad \ \ \ (20)
\end{split}
\end{equation*}
and
\begin{equation*}
\begin{split}
\mu(T_{10})&=x\mu(T_{0}-u)\mu(P_{2})\mu(P_{1})^{r-1}-\mu(P_{2})\mu(P_{1})^{r-1}\sum_{v\in V(T_{0})\atop uv\in  E(T_{0}) }\mu(T_{0}-u-v)
-\mu(T_{0}-u)\mu(P_{1})^{r}\\
&\quad -(r-1)\mu(T_{0}-u)\mu(P_{2})\mu(P_{1})^{r-2}\\
&=x^{2}\mu(T_{0}-u)\mu(P_{1})^{r}-x\mu(T_{0}-u)\mu(P_{1})^{r-1}-x\mu(P_{1})^{r}\sum_{v\in V(T_{0})\atop uv\in  E(T_{0}) }\mu(T_{0}-u-v)\\
&\quad+\mu(P_{1})^{r-1}\sum_{v\in V(T_{0})\atop uv\in  E(T_{0}) }\mu(T_{0}-u-v)-\mu(T_{0}-u)\mu(P_{1})^{r}
-(r-1)x\mu(T_{0}-u)\mu(P_{1})^{r-1}\\
&\quad+(r-1)\mu(T_{0}-u)\mu(P_{1})^{r-2},\qquad \qquad \qquad \qquad \qquad \qquad \qquad \qquad \qquad \qquad \qquad \ \ \ \ \  (21)
\end{split}
\end{equation*}
where the sum ranges over all vertices of $T_{0}$ incident with $u$.\\
\indent By (20) and (21), we have
$$\mu(T_{9})-\mu(T_{10})=-(r-1)\mu(T_{0}-u)\mu(P_{1})^{r-2}+(r-1)\mu(P_{1})^{r-1}\sum_{v\in V(T_{0})\atop uv\in  E(T_{0}) }\mu(T_{0}-u-v).$$
By Lemma \ref{art24}, there exists at least one pendent vertex $v'$ in $T_{0}$ joining vertex $u$ of $T_{0}$. Hence, $\mu(T_{0}-u)=x\mu(T_{0}-u-v')$, which implies that
$$\mu(T_{9})-\mu(T_{10})=(r-1)\sum_{v\in V(T_{0}), v\neq v'\atop uv\in  E(T_{0}) }\mu(P_{1})^{r-1}\mu(T_{0}-u-v).$$
\indent Similar to the proof of Theorem \ref{art32},
$$m(\overline{T_{9}},k)-m(\overline{T_{10}},k)=(r-1)\sum_{v\in V(T_{0}), v\neq v'\atop uv\in  E(T_{0}) }m(\overline{(r-1)P_{1}\cup (T_{0}-u-v)},k-2),\eqno(22)$$
where for every vertex $v$ ($\neq v'$) of $T_{0}$ incident with $u$. Hence $m(\overline{T_{9}},k)\geq m(\overline{T_{10}},k)$. Furthermore, if $k=2$, then $m(\overline{T_{9}},k)-m(\overline{T_{10}},k)\geq1$. So $\overline{T_{9}}\succ\overline{T_{10}}$.
\end{proof}

\begin{definition}\label{art311}
 Suppose that $T_{11}$ is a tree with $n$ vertices and with the edge-independence
number $p$ shown in Figure \ref{fig6}, which has exactly $n-p$ pendent vertices, where
$|V(T_{0})|\geq2$, $s\geq1$ and $t\geq1$. Let $T_{12}$ be the tree with $n$ vertices shown in Figure \ref{fig6}, which is obtained from $T_{11}$.
We designate the transformation from $T_{11}$ to $T_{12}$ as of
type {\bf 6} and denote it by  $\mathcal{F}_{6}$: $T_{11}\rightarrowtail T_{12}$ or $\mathcal{F}_{6}(T_{11})=T_{12}$.
\end{definition}
 \begin{figure}[htbp]
\begin{center}
\includegraphics[scale=1.0]{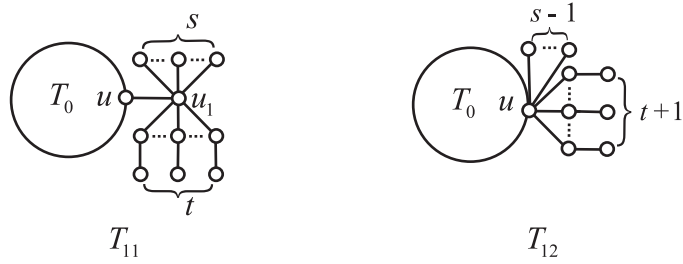}
\caption{\label{fig6}\small
{Two trees $T_{11}$ and $T_{12}$.}}
\end{center}
\end{figure}

\begin{theorem}\label{art312}
Let $T_{11}$ and $T_{12}$ be two trees with $n$ vertices defined in Definition \ref{art311}. Then $\overline{T_{11}}\succ\overline{T_{12}}$.
\end{theorem}
\begin{proof}
Suppose $s>2$. By Lemma \ref{art22},
\begin{equation*}
\begin{split}
\mu(T_{11})&=[x^{2}\mu(P_{1})\mu(P_{2})-sx\mu(P_{2})
-tx\mu(P_{1})^{2}-\mu(P_{1})\mu(P_{2})]\mu(P_{1})^{s-1}\mu(P_{2})^{t-1}\mu(T_{0}-u)\\
&\quad -[x\mu(P_{1})\mu(P_{2})-s\mu(P_{2})
-t\mu(P_{1})^{2}]\mu(P_{1})^{s-1}\mu(P_{2})^{t-1}\sum_{v\in V(T_{0})\atop uv\in  E(T_{0}) }\mu(T_{0}-u-v)
\end{split}
\end{equation*}
and
\begin{equation*}
\begin{split}
\mu(T_{12})&=[x^{2}\mu(P_{1})^{2}\mu(P_{2})-x\mu(P_{1})\mu(P_{2})
-sx\mu(P_{1})\mu(P_{2})+s\mu(P_{2})
-\mu(P_{2})+x\mu(P_{1})\mu(P_{2})\\
&\quad-(t+1)\mu(P_{1})^{2}\mu(P_{2})]\mu(P_{1})^{s-2}\mu(P_{2})^{t-1}\mu(T_{0}-u)\\
&\quad -[x\mu(P_{1})\mu(P_{2})-\mu(P_{2})]\mu(P_{1})^{s-1}\mu(P_{2})^{t-1}\sum_{v\in V(T_{0})\atop uv\in  E(T_{0}) }\mu(T_{0}-u-v),
\end{split}
\end{equation*}
where  the sum ranges over every vertex $v$ of $T_{0}$ adjacent to $u$.\\
\indent Combining the above two   equations, we obtain that
\begin{equation*}
\begin{split}
\mu(T_{11})-\mu(T_{12})&=-[(s+t-1)x^{2}-(s-1)]\mu(P_{1})^{s-2}\mu(P_{2})^{t-1}\mu(T_{0}-u)
\\
&\quad+[(s+t-1)x^{2}-(s-1)]\mu(P_{1})^{s-1}\mu(P_{2})^{t-1}\sum_{v\in V(T_{0})\atop uv\in  E(T_{0}) }\mu(T_{0}-u-v).
\end{split}
\end{equation*}
\indent By Lemma \ref{art24}, there exists at least one pendent vertex $v'$ of $T_{0}$ adjacent to $u$. Hence $\mu(T_{0}-u)=\mu(T_{0}-u-v')$. Thus, simplifying  the above equation, we have
$$\mu(T_{11})-\mu(T_{12})=[(s+t-1)\mu(P_{1})^{s-1}\mu(P_{2})^{t}+t\mu(P_{1})^{s-1}\mu(P_{2})^{t-1}]\sum_{v\in V(T_{0}),v\neq v'\atop uv\in  E(T_{0}) }\mu(T_{0}-u-v).\eqno(23)$$
\indent  As in the proof of Theorem \ref{art31}, we can show that if $s\geq 1$, then
\begin{equation*}
\begin{split}
m(\overline{T_{11}},r)-m(\overline{T_{12}},r)&=(s+t-1)\sum_{v\in V(T_{0}),v\neq v'\atop uv\in  E(T_{0}) }m(\overline{(s-1)P_{1}\cup tP_{2}\cup(T_{0}-u-v)},r-2)\\
&\quad +t\sum_{v\in V(T_{0}),v\neq v'\atop uv\in  E(T_{0}) }m(\overline{(s-1)P_{1}\cup (t-1)P_{2}\cup(T_{0}-u-v)},r-3),
\end{split}
\end{equation*}
which implies that $m(\overline{T_{11}},r)\geq m(\overline{T_{12}},r)$. By the above equation, if $r=2$, then $m(\overline{T_{11}},r)-m(\overline{T_{12}},r)\geq 1$. By (2) and the properties as above, we have $\overline{T_{11}}\succ\overline{T_{12}}$.
\end{proof}

\section{Proofs of Theorems \ref{art01}, \ref{art02} and  \ref{art04}}
\noindent {\bf Proof of Theorem \ref{art01}.}\ \
We  prove that if $T\ncong P_{n}$ then $ME(\overline{T})<ME(\overline{P_{n}})$.
By repeated applications of transformation ${\bf 1}$ in Definition \ref{art31}, we can transform $T$ into $P_{n}$, that is, there exist trees $T^{(i)}$ for $0\leq i\leq l$ such that
$$T=T^{(0)}\hookrightarrow T^{(1)}\hookrightarrow T^{(2)}\hookrightarrow ... \hookrightarrow T^{(l-1)}\hookrightarrow T^{(l)}=P_{n},\eqno(24)$$
where $T^{(l-1)}\neq P_{n}$. By Theorem \ref{art32}, we have
$$\overline{P_{n}}=\overline{T^{(l)}}\succ\overline{T^{(l-1)}}\succ...\succ\overline{T^{(2)}}
\succ\overline{T^{(1)}}\succ\overline{T}.$$
By (4), we obtain immediately the result as follows:
$$ME(\overline{P_{n}})=ME(\overline{T^{(l)}})>ME(\overline{T^{(l-1)}})>...>ME(\overline{T^{(2)}})
>ME(\overline{T^{(1)}})>ME(\overline{T}).$$
\indent By the transformation {\bf 1} in Definition \ref{art31}, Theorem \ref{art32} and (4), it is clear that $$ME(\overline{P_{n}})>ME(\overline{T_{n,2}}).$$
\indent Now we show that $ME(\overline{T_{n,2}})> ME(\overline{T})$. Suppose $T\neq T_{n,2}$. In (24), we know that if $T^{(1-1)}=T_{n,2}$, then $\overline{T_{n,2}}\succ\overline{T}$, which implies $ME(\overline{T^{(n,2)}})>ME(\overline{T})$. If $T^{(1-1)}\neq T_{n,2}$, then $T^{(1-1)}$ must have the from of $T_{3}$ in Fig. \ref{fig2}. By repeated applications of the transformations {\bf 2} and {\bf 3} in Definitions \ref{art33} and \ref{art35}, $T_{3}$ can be transformed into $T_{n,2}$. By Theorems \ref{art34} and \ref{art36}, we have $\overline{T_{n,2}}\succ\overline{T_{3}}\succ\overline{T}$. By (4), $ME(\overline{T_{n,2}})>ME(\overline{T})$.  This completes the proof. \ \ \ \ \ \ \ \ \ \ \ \ \ \ \ \ \ \ \ \ \ \ \ \ \ \ \ \ \ \ \  $\square$

\indent In order to prove Theorem \ref{art02}, we need the following two lemmas by Yan et al. in \cite{yan}.
\begin{lemma}\label{art41}{\rm(\cite{yan}) }
For an arbitrary tree $T$ with $n$ vertices and with edge-independence
number $\nu(T) = p$, if the number of pendent vertices of $T$ is less than $n-p$, then, by
repeated applications of the transformation {\bf 4} in Definition \ref{art37}, $T$ can be transformed
into a tree $T'$ with $n$ vertices and with $\nu(T') = p$, the number of pendent vertices of which
is exactly $n-p$.
\end{lemma}

\begin{lemma}\label{art42}{\rm(\cite{yan}) }
For an arbitrary tree $T$ with $n$ vertices and with  $\nu(T) > p$,
repeated applications of the transformation {\bf 4} in Definition \ref{art37} transform $T$ into a tree $T''$ with $n$ vertices and with $\nu(T'') = p$, the number of pendent vertices of which
is exactly $n-p$.
\end{lemma}
\noindent {\bf Proof of Theorem \ref{art02}.}\ \
Assume $T\ncong T_{n}^{p}$. Now we prove $ME(\overline{T})> ME(\overline{T_{n}^{p}})$ and distinguish the following three cases.\\
{\bf Case 1} We assume that the edge-independence number of $T$ is $p$ and it has exactly $n-p$ pendent vertices. By Lemma \ref{art24}, the structure of $T$ is clear. It is not difficult that, with repeated applications of the transformations {\bf 5} and {\bf 6} in Definitions \ref{art39} and \ref{art311},  $T$ can transformed into $T_{n}^{p}$.
Furthermore, by Theorems \ref{art310} and \ref{art312}, we have $\overline{T}\succ\overline{T_{n}^{p}}$. This indicates, by (4), that $ME(\overline{T})> ME(\overline{T_{n}^{p}})$.

\noindent {\bf Case 2} Assume $\nu(T)=p$ and the number of pendent vertices of $T$ is less than $n-p$. By Lemma \ref{art41}, it can be known that $T$ can be transformed into one tree $T'$ with $n$ vertices, $\nu(T')=p$ and the number of pendent vertices of which is exactly $n-p$. If $T'\neq T_{n}^{p}$, then, by Theorem \ref{art38}, we have $\overline{T}\succ\overline{T'}$. By Case 1,we note that $\overline{T}\succ\overline{T_{n}^{p}}$. If $T'= T_{n}^{p}$, then, by Remark 2, we have $\overline{T}\succ\overline{T'}$. Similarly, by Case 1, we have $\overline{T}\succ\overline{T'}\succ\overline{T_{n}^{p}}$, which implies  $\overline{T}\succ\overline{T_{n}^{p}}$. These mean, by (4), that $ME(\overline{T})> ME(\overline{T_{n}^{p}})$.

\noindent {\bf Case 3} Suppose $\nu(T)>p$. By Lemma \ref{art42}, we know that $T$ can be transformed into one tree $T''$ with $n$ vertices, $\nu(T'')=p$ and the number of pendent vertices of which is exactly $n-p$. Similar to Case 2, we can show that $ME(\overline{T})> ME(\overline{T_{n}^{p}})$.

\indent Combining Case 1-3, The Theorem \ref{art02} holds.\ \ \ \ \ \ \ \ \ \ \ \ \ \ \ \ \ \ \ \ \ \  \ \ \ \ \ \ \ \ \ \ \ \ \ \ \ \ \ \ \ \ \ \ \ \ \ \ \ \ \ \ \ \ \ \ \ \ \ \ \ \ $\square$

\noindent {\bf Proof of Theorem \ref{art04}.}\ \
By Theorems \ref{art01} and \ref{art02},
it can be seen that $ME(\overline{T_{n}^{\frac{n}{2}}})<ME(\overline{T})< ME(\overline{P_{n}})$ and $ME(\overline{T_{n,2}^{1}})< ME(\overline{P_{n}})$. The following we prove that $ME(\overline{T})\leq ME(\overline{T_{n,2}^{1}})$ when $T\in\mathscr{T}_{n,\frac{n}{2}}$ and $T\ncong P_{n}$.

Assume $T\ncong P_{n}$. Similar to the proof of Theorem \ref{art01}, there exist trees $T^{(i)}$ for $0\leq i\leq l$ such that
$$\overline{P_{n}}=\overline{T^{(l)}}\succ\overline{T^{(l-1)}}
\succ\overline{T^{(l-2)}}\succ\overline{T^{(l-3)}}\succ...\succ\overline{T^{(2)}}
\succ\overline{T^{(1)}}\succ\overline{T}.$$
Obviously, $T^{(l-2)}=T_{n,2}^{1}$ or $T^{(l-2)}$ has the from of $T_{3}$ in Fig. \ref{fig2}. By (4), we know that if $T^{(l-2)}=T_{n,2}^{1}$, then $ME(\overline{T})<ME(\overline{T_{n,2}^{1}})$. If $T^{(l-2)}\neq T_{n,2}^{1}$, by repeated applications of the transformations {\bf 2} and {\bf 3} in Definitions \ref{art33} and \ref{art35}, $T_{3}$ can be transformed into $T_{n,2}^{1}$. By Theorems \ref{art34} and \ref{art36}, we have
$\overline{T_{n,2}^{1}}\succ\overline{T}$. By (4), $ME(\overline{T_{n,2}^{1}})>ME(\overline{T})$. \ \ \ \ \ \ \ \ \ \ \ \ \ \ \ \ \ \ \ \ \ \ \ \ \ \ \ \ \ \ \ \ \ \ \ \ \ \ \ \ \ \ $\square$

\begin{remark}
Denote by $\mathscr{T}_{n,p}$ the proper set of $\mathcal{T}_{n,p}$ containing all trees with edge-independence number $p$. Examining Theorem \ref{art04}, we see that if $p=\frac{n}{2}$ in $\mathscr{T}_{n,p}$, then $\overline{P_{n}}$ and $\overline{T_{n,2}^{1}}$ have the maximal and second-maximal matching energy, respectively. A natural question is how to characterize the trees with edge-independence number $p$ whose complements have the maximum matching energy in complements of all trees with edge-independence number $p$.
\end{remark}

\end{document}